\documentclass[copyright,creativecommons]{eptcs}

\usepackage{iftex}
\usepackage{underscore}
\usepackage[T1]{fontenc}

\usepackage{amsmath}
\usepackage{amsthm}
\usepackage{mathtools}

\usepackage[inline]{enumitem}
\usepackage[utf8]{inputenc} %allows non-ascii in bib file

\usepackage[capitalize]{cleveref}
\usepackage[style=eptcs,related=false,doi=true]{biblatex} % [backend=biber,backref=true,maxbibnames=10,style=alphabetic]

\usepackage{tikz}
\usepackage{tikz-cd}

\usepackage{amssymb}
\usepackage{wasysym}
\usepackage{float}
\usepackage{dutchcal} % lowercase cal
\usepackage{mathbbol} % numeric bbl
    \DeclareSymbolFontAlphabet{\mathbbl}{bbold}

  \crefformat{enumi}{\card#2#1#3}
  \crefalias{section}{section}

  \hypersetup{final}
  \hypersetup{hidelinks}

  \setlist{nosep}
  \setlistdepth{6}

  \usetikzlibrary{ 
  	cd,
  	math,
  	decorations.markings,
		decorations.pathreplacing,
  	positioning,
  	arrows.meta,
  	shapes,
		shadows,
		shadings,
  	calc,
  	fit,
  	quotes,
  	intersections
  }

\makeatletter
\let\tikz@lib@matrix@start@cell=\tikz@lib@matrix@normal@start@cell
\makeatother

\theoremstyle{definition}
\newtheorem{definitionx}{Definition}[section]
\newtheorem{examplex}[definitionx]{Example}

\theoremstyle{plain}

\newtheorem{theorem}[definitionx]{Theorem}
\newtheorem{proposition}[definitionx]{Proposition}

\newtheorem{lemma}[definitionx]{Lemma}

\newtheorem*{theorem*}{Theorem}
\newtheorem*{proposition*}{Proposition}
\newtheorem*{corollary*}{Corollary}
\newtheorem*{lemma*}{Lemma}
\newtheorem*{warning*}{Warning}

\newenvironment{example}
  {\pushQED{\qed}\examplex}
  {\popQED\endexamplex}

  \newenvironment{definition}
  {\pushQED{\qed}\definitionx}
  {\popQED\enddefinitionx}

\DeclareSymbolFont{stmry}{U}{stmry}{m}{n}
\DeclareMathSymbol\fatsemi\mathop{stmry}{"23}

\DeclareFontFamily{U}{mathx}{\hyphenchar\font45}
\DeclareFontShape{U}{mathx}{m}{n}{
      <5> <6> <7> <8> <9> <10>
      <10.95> <12> <14.4> <17.28> <20.74> <24.88>
      mathx10
      }{}
\DeclareSymbolFont{mathx}{U}{mathx}{m}{n}
\DeclareFontSubstitution{U}{mathx}{m}{n}
\DeclareMathAccent{\widecheck}{0}{mathx}{"71}

\newcommand{\eps}{\varepsilon}

\newcommand{\cat}[1]{\mathbbl{#1}} % a/the category {symbol} (depends on mathbbol)
\newcommand{\ecat}[1]{\mathcal{#1}} % a/the enriched category {symbol}
\newcommand{\Cat}[1]{\mathbf{#1}} % the category {name}

\renewcommand{\c}{\colon}

\renewcommand{\phi}{\varphi}
\newcommand{\Sig}{\Sigma}

\newcommand{\ol}{\overline}

\newcommand{\inj}{\hookrightarrow}
\newcommand{\cof}{\nrightarrow}

\newcommand{\To}[1]{\xrightarrow{#1}}
\newcommand{\tri}{\mathbin{\triangleleft}}
\newcommand{\tripow}[1]{^{\tri\,#1}}

\newcommand{\ocirc}{\circledcirc}

\newcommand{\biglcoc}[2]{
     \begin{bmatrix}{\vphantom{f_f^f}#2} \\ {\vphantom{f_f^f}#1} \end{bmatrix}
}
\newcommand{\littlelcoc}[2]{
     \begin{bsmallmatrix}{\vphantom{f}#2} \\ {\vphantom{f}#1} \end{bsmallmatrix}
}
\newcommand{\lcoc}[2]{
  \relax\if@display
     \biglcoc{#1}{#2}
  \else
     \littlelcoc{#1}{#2}
  \fi
}

\DeclareMathOperator{\Ob}{Ob}
\newcommand{\op}{^\mathsf{op}}
\DeclareMathOperator{\id}{id}
\DeclareMathOperator{\cod}{cod}

\newcommand{\iso}{\cong}

\newcommand{\1}{\mathsf{1}}

\newcommand{\nn}{\mathbb{N}}

\newcommand{\sets}{\Cat{Set}}

\newcommand{\poly}{\Cat{Poly}}

\newcommand{\dial}{\Cat{Dial}}

\newcommand{\hmg}{\Cat{Hmg}}

\newcommand{\yon}{\mathcal{y}} % (depends on dutchcal)

\allowdisplaybreaks
\begin{document}

\title{Monoidal Structures on Generalized Polynomial Categories}

\def\titlerunning{Monoidal Structures on Generalized Polynomial Categories}

\author{
    Joseph Dorta
    \institute{
        Louisiana State University \\
        Baton Rouge, LA, USA
    } 
    \email{jdorta1@lsu.edu}
        \and  
    Samantha Jarvis
    \institute{
        CUNY Graduate Center \\
        New York, NY, USA
    }
    \email{sjarvis@gradcenter.cuny.edu}
        \and
    Nelson Niu
    \institute{
        University of Washington \\
        Seattle, WA, USA
    }
    \email{nsniu@uw.edu}
}

\def\authorrunning{J.\ Dorta, S.\ Jarvis, and N.\ Niu}

\date{\vspace{-.2in}}

\maketitle

\begin{abstract}
Recently, there has been renewed interest in the theory and applications of de~Paiva’s dialectica categories and their relationship to the category of polynomial functors. Both fall under the theory of generalized polynomial categories, which are free coproduct completions of free product completions of (monoidal) categories.
Here we extend known monoidal structures on polynomial functors and dialectica categories to generalized polynomial categories.
We highlight one such monoidal structure, an asymmetric operation generalizing composition of polynomial functors, and show that comonoids with respect to this structure correspond to categories enriched over a related free coproduct completion. Applications include modeling compositional bounds on dynamical systems.
\end{abstract}

\section{Introduction}

Categories whose morphisms (often referred to as \textit{lenses}) model bidirectional data flows are ubiquitous in applied category theory, with applications to such diverse fields as logic \cite{depaiva,dial-logic}, database management \cite{lens-fibr,alg-lens,comonad-cofunc}, game theory \cite{open-game,comp-game,diegetic}, dynamical and distributed systems \cite{spivak2020poly,2cat-dyn-sys,petri,smithe2022open,dyn-cat}, and machine learning \cite{lens-learn,backprop,grad-learn}.
Moss observed that we can obtain a general class of such categories via free product and coproduct completions, universal constructions with convenient concrete characterizations \cite{moss}.
That is, starting from a category $\cat{C}$, we can form a category $\Sig\Pi\cat{C}$ whose objects are formal coproducts of products of objects in $\cat{C}$, or \textit{polynomials} in $\cat{C}$ for short; then the morphisms between these coproducts of products naturally have both a forward component and a backward component in addition to subsuming the original morphisms from $\cat{C}$.
Examples of such generalized polynomial categories include the category $\poly$ of polynomial functors, which may be used to model interaction protocols \cite{polybook}; and a category whose \textit{homogeneous} polynomials span a full subcategory equivalent to de~Paiva's dialectica category on sets, a model for intuitionistic linear logic \cite{depaiva}.
We review the construction of $\Sig\Pi\cat{C}$ and exhibit these examples in \cref{sec.cons}.

The utility of these examples lies not only in their bidirectional morphisms but also in the assorted ways in which such morphisms can be combined via monoidal products.
There are several ways to lift a monoidal structure on $\cat{C}$ to a monoidal structure on $\Sig\Pi\cat{C}$.
We present two such ways in \cref{sec.monoidal}---one classical, given by an iterated Day convolution \cite{Day}; and one we believe is new in the literature, generalizing functor composition in $\poly$.

Many applications of polynomial functors (such as those in \cite{polybook}) depend on a remarkable result by Ahman and Uustalu \cite{poly-comonad,comonad-cofunc}: the category of comonoids in $\poly$ with respect to the composition product is equivalent to the category whose objects are small categories and whose morphisms are \textit{cofunctors}, as introduced by Aguiar \cite{cofunctor}.
Our main result, \cref{thm}, is that Ahman and Uustalu's statement naturally generalizes to $\Sig\Pi\cat{C}$.
By replacing $\poly$ with $\Sig\Pi\cat{C}$ equipped with our generalized composition product, comonoids become small \textit{enriched} categories whose base of enrichment is $\Sig\cat{C}\op$ with Day convolution.
Then morphisms of these comonoids generalize cofunctors to the enriched setting in a way that coincides with Clarke and Di Meglio's recent definition of \textit{enriched cofunctors} \cite{enr-cof}.
We review the necessary definitions before presenting this correspondence in \cref{sec.com} via an explicit construction.

In \cref{sec.app}, we take $\cat{C}$ to be the extended nonnegative reals to demonstrate how morphisms in $\Sig\Pi\cat{C}$ may be used to model dynamical systems with boundedness conditions preserved by the generalized composition product.
Such morphisms can be lifted to enriched cofunctors via a right adjoint to the forgetful functor from comonoids to their underlying objects; we review a few examples before stating the general result as \cref{thm.cofree}.
Finally, we suggest directions for future work in \cref{sec.future}.

\subsection*{Acknowledgments}
The authors are indebted to the mentorship of Valeria de~Paiva at the 2022 AMS MRC and to insight and feedback from our fellow mentees: Charlotte Aten, Colin Bloomfield, Eric Bond, Matteo Capucci, Bruno Gavranovi\'c, J\'er\'emie Koenig, Abdullah Malik, Francisco Rios, Jan Rooduijn, and Jonathan Weinberger.
Additionally, the authors are grateful for the comments provided by the anonymous reviewers.

This material is based upon work supported by the National Science Foundation under Grant Number DMS 1641020.
Any opinions, findings, and conclusions or recommendations expressed in this material are those of the authors and do not necessarily reflect the views of the National Science Foundation.

\section{Free (co)product completions and polynomial categories} \label{sec.cons}

We begin by recalling two constructions on a category $\cat{C}$: the free product completion and its dual, the free coproduct completion.
Here we follow Moss \cite{moss}; we omit proofs for standard results.

\begin{definition}
The \textbf{\textit{free product completion}} 
of a category $\cat{C}$ is the category $\Pi\cat{C}$, where
\begin{itemize}
    \item an object, denoted $\prod_{i\in I}c_i$, consists of
    \begin{itemize}
        \item a set $I$;
        \item for each $i\in I$, an object $c_i$ in $\cat{C}$;
    \end{itemize}
    \item a morphism $\phi\c\prod_{i\in I}c_i\to\prod_{j\in J}d_j$ consists of
    \begin{itemize}
        \item a function $\phi^\sharp\c J\to I$; 
        \item for each $j\in J$, a morphism $\phi_j\c c_{\phi^\sharp j}\to d_j$ in $\cat C$. \qedhere
    \end{itemize}
\end{itemize}
\end{definition}

The category $\cat{C}$ embeds into $\Pi\cat{C}$ as a full subcategory via $c\mapsto\Pi_{\ast\in\1} c$, where $\1\coloneqq\{*\}$ is the singleton set.
As implied by the name ``free product completion,'' the category $\Pi\cat{C}$ equipped with the embedding $\cat{C}\inj\Pi\cat{C}$ is universal among categories $\cat{D}$ with small products equipped with functors $\cat{C}\to\cat{D}$.

We may alternatively characterize $\Pi\cat{C}$ as follows, using the fact that $[\cat{C},\sets]\op$ equipped with the Yoneda embedding $\cat{C}\inj[\cat{C},\sets]\op$ is the free limit completion of $\cat{C}$ and restricting to products.

\begin{proposition} \label{prop.free-prod-presheaf}
The category $\Pi\cat C$ is equivalent to the full subcategory of $[\cat{C},\sets]\op$ spanned by products of representable functors.
\end{proposition}

\begin{definition}
The \textbf{\textit{free coproduct completion}} 
of a category $\cat C$ is the category $\Sig\cat C$, where
\begin{itemize}
    \item an object, denoted $\sum_{i\in I}c_i$, consists of
    \begin{itemize}
        \item a set $I$;
        \item for each $i\in I$, an object $c_i$ in $\cat{C}$;
    \end{itemize}
    \item a morphism $\phi\c\sum_{i\in I}c_i\to\sum_{j\in J}d_j$ consists of
    \begin{itemize}
        \item a function $\phi\c I\to J$;
        \item for each $i\in I$, a morphism $\phi_i\c c_i\to d_{\phi i}$ in $\cat C$. \qedhere
    \end{itemize}
\end{itemize}

\end{definition}

There is a fully faithful functor $\cat C\inj\Sig\cat C$ sending $c\mapsto\sum_{\ast\in\1}c$.
Comparing the definitions, we find that $(\Sig\cat{C}\op)\op\approx\Pi\cat{C}$; in particular, dualizing \cref{prop.free-prod-presheaf} yields the following.

\begin{proposition}\label{prop.free-coprod-presheaf} The category $\Sig\cat C$ is equivalent to the full subcategory of $[\cat C\op,\sets]$ spanned by coproducts of representable functors.
\end{proposition}

The category $\Sig\cat{C}$ equipped with the embedding $\cat{C}\inj\Sig\cat{C}$ is universal among categories $\cat{D}$ with small coproducts equipped with functors $\cat{C}\to\cat{D}$; in particular $\Sig\cat{C}$ has small coproducts.
As for products in $\Sig\cat{C}$, we have the following proposition.

\begin{proposition} \label{prop.free-sum-gets-prods}
If $\cat C$ has all small products, then $\Sig\cat C$ has all small products given by a distributive law
\begin{equation} \label{eqn.distrib}
    \prod_{i\in I}\sum_{j\in J_i}c_{i,j}\iso\sum_{\ol j\in\prod\limits_{i\in I}J_i}\prod_{i\in I}c_{i,\ol j_i}.
\end{equation}
\end{proposition}
\begin{proof}
\cref{eqn.distrib} holds in $\sets$, so since (co)products are computed pointwise in $[\cat C\op,\sets]$, it holds there as well.
When every $c_{i,j}$ is representable, the right hand side is a coproduct of representables, as products of representables are themselves representable.
Hence \cref{eqn.distrib} also holds in the full subcategory of $[\cat C\op,\sets]$ spanned by coproducts of representables.
Then the conclusion follows from \cref{prop.free-coprod-presheaf}.
\end{proof}

If we freely add products, then freely add coproducts, we obtain the central construction of this paper.

\begin{definition}
The category $\Sig\Pi\cat{C}$ of \textbf{\textit{polynomials in $\cat{C}$}} is the category where
\begin{itemize}
    \item an object, denoted $\sum_{i\in I}\prod_{a\in A_i}c_{i,a}$, consists of
    \begin{itemize}
        \item a set $I$ of \textbf{positions};
        \item for each $i\in I$, a set $A_i$ of \textbf{directions} at $i$;
        \item a doubly-indexed family $(c_{i,a})_{i \in I, a \in A}$ of objects of $\cat{C}$, called \textbf{predicates};
    \end{itemize}
    \item a morphism $\phi\c\sum_{i\in I}\prod_{a\in A_i}c_{i,a}\to\sum_{j\in J}\prod_{b\in B_j}d_{j,b}$ consists of
    \begin{itemize}
        \item an \textbf{on-positions function} $\phi\c I\to J$;
        \item for each $i\in I$, an \textbf{on-directions function} $\phi^\sharp_i\c B_{\phi i}\to A_i$;
        \item for each $i\in I$ and $b\in B_{\phi i}$, an \textbf{on-predicates map} $\phi_{i,b}\c c_{i,\phi^\sharp_ib}\to d_{\phi i,b}$. \qedhere
    \end{itemize}
\end{itemize}
\end{definition}

Unraveling the definitions, we see that $\Sig\Pi\cat{C}$ is indeed the free coproduct completion of the free product completion of $\cat{C}$.
The following characterization of the hom-sets of $\Sig\Pi\cat{C}$ is immediate.

\begin{proposition}
The hom-sets of $\Sig\Pi\cat{C}$ are given by
\[
    \Sig\Pi\cat{C}\!\Big(\sum_{i\in I}\prod_{a\in A_i}c_{i,a},\sum_{j\in J}\prod_{b\in B_j}d_{j,b}\Big)\iso\prod_{i\in I}\sum_{j\in J}\prod_{b\in B_j}\sum_{a\in A_i}\cat{C}(c_{i,a},d_{j,b}).
\]
\end{proposition}

Even though we added products before we added coproducts, \cref{prop.free-sum-gets-prods} ensures that $\Sig\Pi\cat C$ has small products in addition to having small coproducts and that these products distribute over coproducts.

The construction of $\Sig\Pi\cat{C}$ generalizes two particularly versatile categories: the category of \textit{polynomial functors} and one of de~Paiva's \textit{dialectica categories} \cite{depaiva}.
In the remainder of this section, we review each of these categories in turn, observing how they arise from categories of polynomials.

\subsection*{The category of polynomial functors}

We consider $\Sig\Pi\cat{C}$ a \textit{generalized polynomial category} because it generalizes the category $\poly$ of polynomial functors, which we recall below.

\begin{definition}
A \textbf{\textit{polynomial functor}} $p\colon\sets\to\sets$ is a coproduct of representable functors.
That is, there exist $I\in\sets$ and $p[i]\in\sets$ for each $i\in I$ such that, for $\yon^{p[i]}\coloneqq\sets(p[i],-)$,
\[
    p\iso\sum_{i\in I}\yon^{p[i]}.
\]
We call the elements of $p(\1)\iso I$ the \textbf{positions} of $p$ and the elements of $p[i]$ the \textbf{directions} of $p$ at $i$.%
\footnote{The ``positions'' and ``directions'' terminology for polynomial functors was introduced by Spivak \cite{spivak2020poly}.}
We denote the category of polynomial functors and the natural transformations between them by $\poly$.
\end{definition}

It turns out that $\poly$ is the category of polynomials in the terminal category $\cat{1}$, consisting of one object and no non-identity morphisms.

\begin{proposition}
$\poly\approx\Sig\Pi\cat1$.
\end{proposition}

\begin{proof}
By definition, $\Pi\cat1\approx\sets\op$.
Then \cref{prop.free-coprod-presheaf} implies that $\Sig\Pi\cat1$ is the full subcategory of $[\sets,\sets]$ spanned by coproducts of representables.
\end{proof}

Viewing $\poly$ as $\Sig\Pi\cat{1}$, we can characterize the morphisms of $\poly$ as follows.

\begin{example}
A morphism $\phi\c p\to q$ in $\poly\approx\Sig\Pi\cat1$ consists of
\begin{itemize}
    \item an \textbf{on-positions function} $\phi_\1\c p(\1)\to q(\1)$;%
    \footnote{We use a subscript $\1$ for the on-positions function as it is the $\1$-component of $\phi$ as a natural transformation \cite{polybook}.}
    \item for each $i\in p(\1)$, an \textbf{\textit{on-directions function}} $\phi^\sharp_i\c q[\phi i]\to p[i]$. \qedhere
\end{itemize}
\end{example}

\subsection*{The dialectica category on sets}

Rather than working with the entire category of polynomials in $\cat{C}$, it is sometimes easier to work with one of its full subcategories, which we define below.

\begin{definition}
A polynomial in $\cat{C}$ is \textbf{\textit{homogeneous}}%
\footnote{The terminology comes from algebra, where a \textit{homogeneous} polynomial is one whose summands all have the same degree.}
if it can be written in the form
\[
    \sum_{i \in I} \prod_{a \in A} u_{i,a},
\]
where the set $A$ does not depend on $i\in I$.
We let $\hmg(\cat{C})$ denote the full subcategory of $\Sig\Pi\cat{C}$ spanned by homogeneous polynomials.
\end{definition}

As an example, let $\cat2$ denote the walking arrow category, which has two objects $\bot$ and $\top$ and one non-identity arrow $\bot\to\top$.

\begin{example}
In the category $\hmg(\cat2)$,
\begin{itemize}
    \item an object, denoted $\sum_{i\in I}\prod_{a\in A}c_{i,a}$, consists of
    \begin{itemize}
        \item two sets, $I$ and $A$;
        \item for each $(i,a)\in I\times A$, an object $c_{i,a}\in\{\bot,\top\}$;
    \end{itemize}
    \item a morphism $\phi\c\sum_{i\in I}\prod_{a\in A}c_{i,a}\to\sum_{j\in J}\prod_{b\in B}d_{j,b}$ consists of
    \begin{itemize}
        \item a function $\phi\c I\to J$;
        \item a function $\phi^\sharp\c I\times B\to A$; such that
        \item for each $i\in I$ and $b\in B$, if $c_{i,\phi^\sharp(i,b)}=\top$, then $d_{\phi i,b}=\top$. \qedhere
    \end{itemize}
\end{itemize}
\end{example}

This is precisely de~Paiva's original dialectica category on $\sets$ \cite{depaiva}.

\begin{proposition}
    $\hmg(\cat{2}) \approx \dial(\sets)$.
\end{proposition}

\section{Monoidal structures on polynomial categories} \label{sec.monoidal}

Most of the applications of $\poly$ and $\dial(\sets)$ rely on their monoidal structures; in this section, we will generalize such structures to $\Sig\Pi\cat{C}$.
Throughout, let $(\cat{C},e,\cdot)$ be a monoidal category with unit $e\in\cat{C}$ and product $\cdot\c\cat{C}\times\cat{C}\to\cat{C}$.
The monoidal structure on $\cat{C}$ then induces monoidal structures on $\Sig\Pi\cat{C}$.

\subsection{The parallel product}

A monoidal product $\cdot$ on $\cat{C}$ always induces a monoidal product $\odot$ on the free colimit completion $[\cat{C}\op,\sets]$ of $\cat{C}$: the Day convolution \cite{Day}, which agrees with $\cdot$ on the full subcategory $\cat{C}\inj[\cat{C}\op,\sets]$.

\begin{proposition}\label{parcon1}
The Day convolution $\odot$ on $[\cat{C}\op,\sets]$ restricts to a monoidal product on the free coproduct completion $\Sig\cat{C}$ of $\cat{C}$, yielding a distributive monoidal category $(\Sig\cat{C},e,\odot)$.
\end{proposition}

\begin{proof}
The Day convolution is a coend construction and thus preserves coproducts.
Hence $\Sig\cat{C}$ is closed under $\odot$, and $\odot$ distributes over coproducts:
\begin{equation} \label{eq.par}
    \Big(\sum_{i\in I}c_i\Big)\odot\Big(\sum_{j\in J}d_j\Big)\iso\sum_{i\in I}\sum_{j\in J}(c_i\odot d_j)\iso\sum_{(i,j)\in I\times J}(c_i\cdot d_j). \qedhere
\end{equation}
\end{proof}

\cref{eq.par} tells us how to evaluate $\odot$ on arbitrary objects in $\Sig\cat{C}$.
We dualize this construction to obtain an analogous monoidal product on $\Pi\cat{C}\approx(\Sig\cat{C}\op)\op$.

\begin{proposition}\label{parcon2}
There is a monoidal structure on $\Pi \cat{C}$ with unit $e$ whose monoidal product $\ocirc$ is given by 
\[
    \Big(\prod_{a\in A}c_a\Big)\ocirc\Big(\prod_{b\in B}d_b\Big)\iso\prod_{(a,b)\in A\times B}(c_a\cdot d_b).
\]
\end{proposition}

Thus, to obtain a monoidal structure on $\Sig\Pi\cat{C}$, we may first lift the monoidal structure on $\cat{C}$ to $\Pi\cat{C}$, then lift the monoidal structure on $\Pi\cat{C}$ to $\Sig\Pi\cat{C}$.

\begin{proposition}
There is a monoidal structure on $\Sig \Pi \cat{C}$ with unit $e$ whose monoidal product, which we call the \textbf{\textit{parallel product}} and denote by $\otimes$, is given by
\[
    \Big(\sum_{i\in I}\prod_{a\in A_i}c_{i,a}\Big)\otimes\Big(\sum_{j\in J}\prod_{b\in B_j}d_{j,b}\Big)\iso\sum_{i\in I}\sum_{j\in J}\prod_{a\in A_i}\prod_{b\in B_j}(c_{i,a}\cdot d_{j,b}).
\]
\end{proposition}

\begin{proof}
Apply \cref{parcon2} on $(\cat{C},e,\odot)$ to obtain $(\Pi\cat{C},e,\ocirc)$, then apply \cref{parcon1} on $(\Pi\cat{C},e,\ocirc)$ to obtain $(\Sig\Pi\cat{C},e,\otimes)$. \end{proof}

\begin{example}
To justify our use of the name ``parallel product," we consider an example.
Let $\cat{C}\coloneqq\cat1$, whose unique object we call $\yon$.
There is a unique monoidal structure on $\cat{C}$ given by $\yon\cdot\yon=\yon$.

Following \cite{polybook}, in $\Sig\Pi\cat{1}\approx\poly$, an object $\sum_{i\in I}\prod_{a\in A_i}\yon$, which we denote by $\sum_{i\in I}\yon^{A_i}$ for short, can be thought of as an \textbf{\textit{interface}}, with a number of possible \textit{positions} from $I$ it could expose and, according to the position $i\in I$ it is currently exposing, a number of possible \textit{directions} from $A_i$ it could receive.
A morphism $\phi\c\sum_{i\in I}\yon^{A_i}\to\sum_{i'\in I'}\yon^{A'_{i'}}$ in $\poly$ can then be viewed as an \textbf{\textit{interaction protocol}} between interfaces.
On positions, $\phi$ converts any position $i\in I$ that the domain could expose to a position $\phi i\in I'$ for the codomain to expose; then on directions, $\phi$ converts any direction $a'\in A'_{\phi i}$ that the codomain could receive to a direction $\phi^\sharp_i a'\in A_i$ for the domain to receive.

Then taking the parallel product of two such interaction protocols yields a single interaction protocol that models the two original protocols simultaneously---or in \textit{parallel}.
More concretely, given interaction protocols $\phi\c\sum_{i\in I}\yon^{A_i}\to\sum_{i'\in I'}\yon^{A'_{i'}}$ and $\psi\c\sum_{j\in J}\yon^{B_j}\to\sum_{j'\in J'}\yon^{B'_{j'}}$, their parallel product $\phi\otimes\psi$ converts a pair of positions $(i,j)\in I\times J$ from its domain
\[
    \Big(\sum_{i\in I}\yon^{A_i}\!\Big)\otimes\Big(\sum_{j\in J}\yon^{B_j}\!\Big)\iso\sum_{(i,j)\in I\times J}(\yon\cdot\yon)^{A_i\times B_j}\iso\sum_{(i,j)\in I\times J}\yon^{A_i\times B_j}
\]
to the pair of positions $(\phi i, \psi j)\in I'\times J'$ from its codomain
\[
    \Big(\sum_{i'\in I'}\yon^{A'_{i'}}\!\Big)\otimes\Big(\sum_{j'\in J'}\yon^{B'_{j'}}\!\Big)\iso\sum_{(i',j')\in I'\times J'}\yon^{A'_{i'}\times B'_{j'}}
\]
by applying the on-positions functions of $\phi$ and $\psi$ in parallel; then converts a pair of directions $(a',b')\in A'_{\phi i}\times B'_{\psi j}$ from its codomain to the pair of directions $(\phi^\sharp_i a', \psi^\sharp_j b')$ from its domain by applying the on-directions functions of $\phi$ and $\psi$ in parallel.
\end{example}

\subsection{The composition product}

Here we introduce another monoidal structure on $\Sig\Pi\cat{C}$ induced by the monoidal product on $\cat{C}$.

\begin{definition}
The \textbf{\textit{composition product}} $\tri$ of two objects in $\Sig\Pi\cat{C}$ is given by
\[
    \Big(\sum_{i\in I}\prod_{a\in A_i}u_{i,a}\Big)\tri\Big(\sum_{j\in J}\prod_{b\in B_j}v_{j,b}\Big)\coloneqq\sum_{i\in I}\prod_{a\in A_i}\sum_{j\in J}\prod_{b\in B_j}(u_{i,a}\cdot v_{j,b}).
    \qedhere
\]
\end{definition}

We call this the \textit{composition} product as it generalizes the composition operation on polynomial functors when $\cat{C}=\cat1$: composing $\sum_{i\in I}\prod_{a\in A_i}\yon$ with $\sum_{j\in J}\prod_{b\in B_j}\yon$ yields the functor $\sum_{i\in I}\prod_{a\in A_i}\sum_{j\in J}\prod_{b\in B_j}\yon$.
Distributivity, as given by \cref{eqn.distrib}, yields the following alternate form for this product.

\begin{lemma} \label{lemma.comp}
The composition product can be rewritten as
\[
    \Big(\sum_{i\in I}\prod_{a\in A_i}u_{i,a}\Big)\tri\Big(\sum_{j\in J}\prod_{b\in B_j}v_{j,b}\Big)\iso\sum_{i\in I}\:\sum_{j\c A_i\to J}\:\prod_{a\in A_i}\prod_{b\in B_{ja}}(u_{i,a}\cdot v_{ja,b}).
\]
\end{lemma}

% TODO: prove assoc?

\begin{proposition}
There is a monoidal category $(\Sig\Pi\cat{C},e,\tri)$.
\end{proposition}
\begin{proof}
Routine, but we will describe the behavior of $\tri$ on morphisms: given \[
    \phi\c\sum_{i\in I}\prod_{a\in A_i}u_{i,a}\to\sum_{k\in K}\prod_{c\in C_k}w_{k,c} \quad\text{and}\quad \psi\c\sum_{j\in J}\prod_{b\in B_j}v_{j,b}\to\sum_{\ell\in L}\prod_{d\in D_\ell}x_{\ell,d},
\]
the morphism
\[
    \phi\tri\psi\c\sum_{i\in I}\:\sum_{j\c A_i\to J}\:\prod_{a\in A_i}\prod_{b\in B_{ja}}(u_{i,a}\cdot v_{ja,b})\to\sum_{k\in K}\:\sum_{\ell\c C_k\to L}\:\prod_{c\in C_k}\prod_{d\in D_{\ell c}}(w_{k,c}\cdot x_{\ell c,d})
\]
(whose domain and codomain we have rewritten using \cref{lemma.comp}) consists of the following data:

\newpage

\begin{itemize}
    \item an \textit{on-positions function} $\phi\tri\psi\c\sum_{i\in I}J^{A_i}\to\sum_{k\in K}L^{C_k}$ consisting of:
    \begin{itemize}
        \item a function $I\to K$ given by $\phi$;
        \item for each $i\in I$, a function $J^{A_i}\to L^{C_{\phi i}}$ given by precomposing $\phi^\sharp_i\c C_{\phi i}\to A_i$ and postcomposing $\psi\c J\to L$;
    \end{itemize}
    \item for each $i\in I$ and $j\c A_i\to J$, sent to $\phi i\in K$ and $\psi j\phi^\sharp_i\c C_{\phi i}\to L$ by the on-positions function, an \textit{on-directions function} $(\phi\tri\psi)^\sharp_{i,j}\c\sum_{c\in C_{\phi i}}D_{\psi j\phi^\sharp_i c}\to\sum_{a\in A_i}B_{ja}$ consisting of:
    \begin{itemize}
        \item a function $C_{\phi i}\to A_i$ given by $\phi^\sharp_i$;
        \item for each $c\in C_{\phi i}$, a function $D_{\psi j\phi^\sharp_i c}\to B_{j\phi^\sharp_i c}$ given by $\psi^\sharp_{j\phi^\sharp_i c}$;
    \end{itemize}
    \item for each $i\in I, j\c A_i\to J, c\in C_{\phi i},$ and $d\in D_{\psi j\phi^\sharp_i c}$, sent to $\phi^\sharp_i\c C_{\phi i}\to A_i$ and $\psi^\sharp_{j\phi^\sharp_i c}\c D_{\psi j\phi^\sharp_i c}\to B_{j\phi^\sharp_i c}$ by the on-directions function, an \textit{on-predicates map} $(\phi\tri\psi)_{i,j,c,d}\c u_{i,\phi^\sharp_i c}\cdot v_{j',\psi^\sharp_{j'} d}\to w_{\phi i,c}\cdot x_{\psi j',d}$ (here $j'\coloneqq j\phi^\sharp_i c$) given by $\phi_{i,c}\cdot\psi_{j',d}$. \qedhere
\end{itemize}
\end{proof}

\section{Composition comonoids as enriched categories} \label{sec.com}

Our main result concerns the category of comonoids in $(\Sig\Pi\cat{C},e,\tri)$.
We will show that it is equivalent to a category whose objects are enriched categories and whose morphisms are enriched cofunctors.
While the former may be more familiar than the latter, we review both these definitions here.

Recall the definition of a \textit{category enriched over a monoidal category} from Kelly \cite{enrich}. We restate it here for the special case where the enriching category is $(\Sig\cat{C}\op,e,\odot)$.

\begin{definition}
A \textbf{\textit{small $(\Sig\cat{C}\op,e,\odot)$-enriched category}} $\ecat{A}$, with $\odot$ defined as in Proposition \ref{parcon1}, consists of the following data:
\begin{itemize}
    \item a set $\Ob\ecat{A}$ (or just $\ecat{A}$) of \textbf{objects};
    \item for each $x,y\in\Ob\ecat{A}$, a \textbf{hom-family} $\sum_{f\c x\to y}|f|\in\Sig\cat{C}\op$ consisting of:
    \begin{itemize}
        \item a set $\ecat{A}(x,y)$ of \textbf{morphisms}, i.e.\ a \textbf{hom-set}, with $f\in\ecat{A}(x,y)$ denoted by $f\c x\to y$;
        \item for each morphism $f\c x\to y$, a \textbf{weight} $|f|\in\cat{C}$;
    \end{itemize}
    \item for each $x\in\Ob\ecat{A}$, a morphism $e\to\sum_{f\c x\to x}|f|$ in $\Sig\cat{C}\op$ consisting of:
    \begin{itemize}
        \item an \textbf{identity morphism} $\id_x\c x\to x$;
        \item an \textbf{identity map} $\eta_x\c|\id_x\!|\to e$ from $\cat{C}$;
    \end{itemize}
    \item for each $x,y,z\in\Ob\ecat{A}$, a morphism
    \[
        \sum_{f\c x\to y}\;\sum_{g\c y\to z}(|f|\cdot|g|) \to \sum_{h\c x\to z}|h|
    \]
    in $\Sig\cat{C}\op$ consisting of, for each $f\c x\to y$ and $g\c y\to z$:
    \begin{itemize}
        \item a \textbf{composite morphism} $gf\c x\to z$;
        \item a \textbf{composite map} $\mu_{f,g}\c|gf|\to|f|\cdot|g|$ from $\cat{C}$.
    \end{itemize}
\end{itemize}
Here $w,x,y,z\in\Ob\ecat{A}$ and $f\c w\to x$, $g\c x\to y$, and $h\c y\to z$ must satisfy the following:

\begin{itemize}
    \item \textbf{unitality}, that $f\id_w = f = \id_x f$ and the following diagram commutes in $\cat{C}$, up to unitors:
    \[
    \begin{tikzcd}[column sep=large]
        |\id_w\!|\cdot|f|
        \ar[dr, "\eta_w\cdot|f|"']
        &
        \ar[l, "\mu_{\id_w,f}"']
        |f|
        \ar[r, "\mu_{f,\id_x}"]
        \ar[d, equals]
        &
        |f|\cdot|\id_x\!|
        \ar[dl, "|f|\cdot\eta_x"]
        \\
        &
        |f|
    \end{tikzcd}
    \]

    \newpage
    
    \item \textbf{associativity}, that $(hg)f=h(gf)$ and the following commutes in $\cat{C}$, up to associators:
    \[
    \begin{tikzcd}[column sep=large]
        |hgf|
        \ar[r, "\mu_{f,hg}"]
        \ar[d, "\mu_{gf,h}"']
        &
        |f|\cdot|hg|
        \ar[d, "|f|\cdot\mu_{g,h}"]
        \\
        |gf|\cdot|h|
        \ar[r, "\mu_{f,g}\cdot|h|"]
        &
        |f|\cdot|g|\cdot|h|
    \end{tikzcd}
    \]
\end{itemize}
\end{definition}

By Proposition \ref{parcon1}, the monoidal category $(\Sig\cat{C}\op,e,\odot)$ is distributive, so there exists a notion of a \textit{$(\Sig\cat{C}\op,e,\odot)$-enriched cofunctor} as introduced by Clarke and Di Meglio \cite{enr-cof}.
We restate the definition of an enriched cofunctor in this special case here.

\begin{definition}
A \textbf{\textit{$(\Sig\cat{C}\op,e,\odot)$-enriched cofunctor} }$\Phi\c\ecat{A}\cof\ecat{B}$ between small $(\Sig\cat{C}\op,e,\odot)$-enriched categories $\ecat{A}$ and $\ecat{B}$ consists of the following data:
\begin{itemize}
    \item a function $\Phi\c\Ob\ecat{A}\to\Ob\ecat{B}$;
    \item for each $a\in\ecat{A}, b\in\ecat{B},$ and morphism $f\c\Phi a\to b$ from $\ecat{B}$:
    \begin{itemize}
        \item a morphism $\Phi^\sharp_af\c a\to x$ from $\ecat{A}$ with $\Phi x = b$;
        \item a morphism $\Phi_{a,f}\c|\Phi^\sharp_af|\to|f|$ from $\cat{C}$.
    \end{itemize}
\end{itemize}
Here $a,x\in\ecat{A}$; $b,b'\in\ecat{B}$; $f\c\Phi a\to b$ with $\Phi^\sharp_af\c a\to x$; and $g\c b\to b'$ must satisfy:
\begin{itemize}
    \item \textbf{preservation of identities}, that $\Phi^\sharp_a(\id_{\Phi a})=\id_a$ and the following commutes in $\cat{C}$:
    \[
    \begin{tikzcd}[column sep=large]
        |\id_a\!| \ar[rd, "\eta_a"'] \ar[r, "\Phi_{a,\id_{\Phi a}}"] & |\id_{\Phi a}\!| \ar[d, "\eta_{\Phi a}"] \\
        & e
    \end{tikzcd}
    \]
    \item \textbf{preservation of composites}, that $\Phi^\sharp_a(gf)=\Phi^\sharp_{x}(g)\Phi^\sharp_a(f)$ and the following commutes in $\cat{C}$:
    \[
    \begin{tikzcd}[column sep=large]
        |\Phi^\sharp_a(gf)| \ar[d, "\Phi_{a,gf}"'] \ar[r, "\mu_{\Phi^\sharp_af,\Phi^\sharp_xg}"] & |\Phi^\sharp_af| \cdot |\Phi^\sharp_xg| \ar[d, "\Phi_{a,f}\cdot\Phi_{x,g}"] \\
        |gf| \ar[r, "\mu_{f,g}"] & |f|\cdot|g|
    \end{tikzcd}
    \]
\end{itemize}   
\end{definition}
There is then a category whose objects are small $(\Sig\cat{C}\op,e,\odot)$-enriched categories and whose morphisms are $(\Sig\cat{C}\op,e,\odot)$-enriched cofunctors.
While enriched cofunctors differ from enriched functors, it is nevertheless the case that an isomorphism in this category corresponds to our usual notion of isomorphism of enriched categories as defined by a pair of invertible enriched functors.

The following is a generalization of a result by Ahman and Uustalu \cite{poly-comonad,comonad-cofunc}: that the category of polynomial comonads is equivalent to the category of small categories and cofunctors, corresponding to the case where $\cat{C}=\cat1$ (the $\sets$-enriched case, for $\Sig\cat1\op\approx\sets$) in the theorem below.

\begin{theorem} \label{thm}
The category of comonoids in the monoidal category $(\Sig\Pi\cat{C},e,\tri)$ is equivalent to the category of small $(\Sig\cat{C}\op,e,\odot)$-enriched categories and enriched cofunctors.
\end{theorem}

\begin{proof}
First, we describe how to construct a comonoid in $(\Sig\Pi\cat{C},e,\tri)$ from each $(\Sig\cat{C}\op,e,\odot)$-enriched category; the inverse construction will then be evident.
Given a small $(\Sig\cat{C}\op,e,\odot)$-enriched category~$\ecat{A}$, define a polynomial in $\cat{C}$ with positions $\Ob\ecat{A}$, directions $A_i\coloneqq\sum_{j\in\ecat{A}}\ecat{A}(i,j)$ for $i\in\ecat{A}$, and predicate $|a|\in\cat{C}$ for $i\in\ecat{A}$ and $(j,a\c i\to j)\in A_i$.
In other words: positions are objects, directions are morphisms of a given domain, and predicates are the morphisms' weights.
We endow this polynomial $\sum_{i\in\ecat{A}}\prod_{a\c i\to\_}|a|$ (where $a\c i\to\_$ denotes a morphism $a$ in $\ecat{A}$ with domain $i$ and arbitrary codomain) with a comonoid structure as follows.
Its counit
\[
    \eps\c\sum_{i\in\ecat{A}}\;\prod_{a\c i\to\_}|a|\to e
\]
is trivial on positions, the assignment $i\mapsto\id_i$ on directions, and the identity map $\eta_i\c|\id_i\!|\to e$ on predicates.
Meanwhile its comultiplication
\[
    \delta\c\sum_{i\in\ecat{A}}\;\prod_{a\c i\to\_}|a|\to\Big(\sum_{i\in\ecat{A}}\;\prod_{b\c i\to\_}|b|\Big)\tri\Big(\sum_{j\in\ecat{A}}\;\prod_{c\c j\to\_}|c|\Big)\iso\sum_{i\in\ecat{A}}\;\sum_{j\c A_i\to\Ob\ecat{A}}\;\prod_{b\c i\to\_}\;\prod_{c\c jb\to\_}|b|\cdot|c|
\]
is the assignment $i\mapsto(i,\cod)$ on positions, where $\cod\c A_i\to\Ob\ecat{A}$ sends each morphism $a\c i\to j$ to its codomain $j$; morphism composition on directions, sending $b\c i\to\_$ and $c\c\cod(b)\to\_$ to $cb\c i\to\_$; and the composite map $\mu_{b,c}\c|cb|\to|b|\cdot|c|$ on predicates.
The counitality and coassociativity of the comonoid follow from the unitality and associativity of the enriched category, as well as the equations $\cod(\id_i)=i$ and $\cod(cb)=\cod(c)$.
Moreover, from any comonoid we can recover its corresponding enriched category up to isomorphism.

Next, we describe how to construct a morphism of comonoids in $(\Sig\Pi\cat{C},e,\tri)$ from each $(\Sig\cat{C}\op,e,\odot)$-enriched cofunctor; again the inverse construction will then be evident.
Given a $(\Sig\cat{C}\op,e,\odot)$-enriched cofunctor $\Phi\c\ecat{A}\cof\ecat{B}$ between small $(\Sig\cat{C}\op,e,\odot)$-enriched categories $\ecat{A}$ and $\ecat{B}$, we construct a structure-preserving morphism
\[
    \phi\c\sum_{i\in\ecat{A}}\prod_{a\c i\to\_}|a|\to\sum_{j\in\ecat{B}}\prod_{b\c j\to\_}|b|
\]
in $\Sig\Pi\cat{C}$ between the comonoids corresponding to $\ecat{A}$ and $\ecat{B}$ like so.
On positions, set $\phi i\coloneqq\Phi i\in\ecat{B}$ for $i\in\ecat{A}$; on directions, set $\phi^\sharp_i b\coloneqq(\Phi^\sharp_ib\c i\to\_)$ in $\ecat{A}$ for $i\in\ecat{A}$ and $b\c\Phi i\to\_$ in $\ecat{B}$; and on predicates, set $\phi_{i,b}\coloneqq(\Phi_{i,b}\c|\Phi^\sharp_ib|\to|b|)$ in $\cat{C}$ for $i\in\ecat{A}$ and $b\c\Phi i\to\_$ in $\ecat{B}$.
That $\phi$ preserves counits and comultiplications follows from the fact that $\Phi$ preserves identities and composites and that $\Phi(\cod(\Phi^\sharp_ia))=\cod(a)$.
Moreover, from any comonoid morphism we can recover its corresponding enriched cofunctor.
\end{proof}

\section{Application: compositional bounds on dynamical systems} \label{sec.app}

Here we give an example of how the structure of $\Sig\Pi\cat{C}$ may be used to model open dynamical systems and their invariants.
This case study is by no means comprehensive; we seek only to hint at the possibilities of how $\Sig\Pi\cat{C}$ may be used.

Throughout, we let $(\cat{C},e,\cdot)\coloneqq([0,\infty]_\leq,0,+)$, the poset of nonnegative extended reals ordered by $\leq$ viewed as a category and endowed with the additive monoidal structure.
We take the free coproduct completion of its opposite category and endow it with a monoidal structure $\oplus$ given by Day convolution.
Then a $(\Sig[0,\infty]_\geq,0,\oplus)$-enriched category is an \textit{additively weighted category} \cite{weight}.

\newpage

\begin{definition}
    An \textbf{\textit{additively weighted category}} (or \textbf{\textit{weighted category}}) $\ecat{X}$ is a small $(\Sig[0,\infty]_\geq,0,\oplus)$-enriched category.
    It thus consists of the following data:
    \begin{itemize}
        \item a set $\Ob\ecat{X}$ of \textbf{objects} or \textbf{points};
        \item for each $x,y\in\Ob\ecat{X}$, an object $\sum_{p\c x\to y}|p|\in\Sig[0,\infty]_\geq$ consisting of:
        \begin{itemize}
            \item a set $\ecat{X}(x,y)$ of \textbf{morphisms} or \textbf{paths}, with $p\in\ecat{X}(x,y)$ denoted by $p\c x\to y$;
            \item for each path $p\c x\to y$, a \textbf{weight} or \textbf{cost} $|p|\in[0,\infty]$;
        \end{itemize}
        \item for each $x\in\Ob\ecat{X}$, a morphism $0\to\sum_{f\c x\to x}|f|$ in $\Sig[0,\infty]_\geq$ consisting of:
        \begin{itemize}
            \item a \textbf{constant path} $\id_x\c x\to x$,
            \item satisfying \textbf{nonpositivity}: $|\id_x\!|\leq0$, and thus $|\id_x\!|=0$;
        \end{itemize}
        \item for each $x,y,z\in\Ob\ecat{X}$, a morphism
        \[
            \sum_{f\c x\to y}\sum_{g\c y\to z}(|f|+|g|) \to \sum_{h\c x\to z}|h|
        \]
        in $\Sig[0,\infty]_\geq$ consisting of, for each $f\c x\to y$ and $g\c y\to z$:
        \begin{itemize}
            \item a \textbf{composite path} $gf\c x\to z$,
            \item satisfying the \textbf{triangle inequality}: $|gf|\leq|f|+|g|$. \qedhere
        \end{itemize}
    \end{itemize}
\end{definition}

A weighted category $\ecat{X}$ with $|\ecat{X}(x,y)|=1$ for all $x,y\in\Ob\ecat{X}$ is a Lawvere metric space \cite{lawvere}.

By \cref{thm}, a weighted category $\ecat{X}$, defined above as an enriched category, is equivalently a comonoid in $(\Sig\Pi[0,\infty]_\leq,0,\tri)$.
Then we can define a discrete dynamical system on $\ecat{X}$ in terms of the category $\Sig\Pi[0,\infty]_\leq$ as follows.

\begin{definition}
    A \textbf{\textit{discrete dynamical system}} on a weighted category $\ecat{X}$, viewed as a comonoid object $\ecat{X}\in\Sig\Pi[0,\infty]_\leq$, is a morphism $\phi\c\ecat{X}\to\infty$ in $\Sig\Pi[0,\infty]_\leq$.
    It thus consists of the following data:
    \begin{itemize}
        \item a trivial \textit{on-positions function} $\phi\c\Ob\ecat{X}\to\1$;
        \item for each point $x\in\Ob\ecat{X}$, an \textit{on-directions function} $\phi^\sharp_x\c\1\to\sum_{y\in\ecat{X}}\ecat{X}(x,y)$ that picks out a path $\phi^\sharp_x$ from $x$ to some other point,
        \item satisfying the trivial inequality $|\phi^\sharp_x|\leq\infty$. \qedhere
    \end{itemize}
\end{definition}
In other words, a discrete dynamical system on $\ecat{X}$ assigns to each point $x$ in $\ecat{X}$ a path $\phi^\sharp_x\c x\to x_1$ out of that point.
The intuition is that starting from $x$, the system moves to a new point $x_1$ along the path $\phi^\sharp_x$ in one time step.
We can ``run'' the system by taking the $n$-fold composition product $\phi\tripow{n}$ for $n\in\nn$ and composing with the canonical $n$-ary comultiplication $\delta^{n-1}$ of $\ecat{X}$ provided by its comonoid structure:%
\footnote{We inductively define $\delta^1\coloneqq\delta$ and $\delta^n\coloneqq(\id_{\ecat{X}\tripow{(n-1)}}\:\tri\:\delta)\circ\delta^{n-1}$.}
\begin{equation} \label{diag.seq}
    \ecat{X}\To{\delta^{n-1}}\ecat{X}\tripow{n}\To{\phi\tripow{n}}\infty\tripow{n}\iso\infty+\cdots+\infty\iso\infty.
\end{equation}
This is a new discrete dynamical system that assigns to each point $x$ in $\ecat{X}$ the $n$-fold composite of paths
\begin{equation} \label{diag.comp-path}
    x\To{\phi^\sharp_x}x_1\To{\phi^\sharp_{x_1}}x_2\To{\phi^\sharp_{x_2}}\cdots\To{\phi^\sharp_{x_{n-1}}}x_n
\end{equation}
from $\ecat{X}$, mapping out the evolution of the dynamical system after $n$ time steps.
Similarly, given another discrete dynamical system $\psi\c\ecat{X}\to\infty$, we can compose it with the first to obtain a third system that runs one before the other:
\[
    \ecat{X}\To\delta\ecat{X}\tri\ecat{X}\To{\phi\:\tri\:\psi}\infty\tri\infty\iso\infty.
\]

Furthermore, we could repackage the data of a discrete dynamical system as an enriched cofunctor by the following proposition, where $\ecat{X}$ is a weighted category viewed as a comonoid object $\ecat{X}\in\Sig\Pi[0,\infty]_\leq$.

\begin{proposition} \label{prop.dyn-sys-cof}
    There is a natural correspondence between discrete dynamical systems $\phi\c\ecat{X}\to\infty$ and enriched cofunctors $\Phi\c\ecat{X}\cof\prod_{n\in\nn}\infty$, whose codomain is the one-object $(\Sig[0,\infty]_\geq,0,\oplus)$-enriched category with hom-set $\nn$, addition as composition, and all weights infinite.
\end{proposition}
\begin{proof}
    Given $\phi$, construct $\Phi$ by setting $\Phi^\sharp_x(n)$ to the composite path defined in \eqref{diag.comp-path} for $x\in\Ob\ecat{X}$ and $n\in\nn$ (with $\Phi^\sharp_x(0)\coloneqq\id_x$); the cofunctor laws follow immediately.
    Conversely, given $\Phi$, construct $\phi$ by setting $\phi^\sharp_x\coloneqq\Phi^\sharp_x(1)$.
    These constructions are natural and mutually inverse.
\end{proof}

Thus discrete dynamical systems on $\ecat{X}$ are precisely enriched cofunctors $\ecat{X}\cof\prod_{n\in\nn}\infty$.
We could generalize how these systems run by replacing $\nn$ with some other monoid, or indeed by replacing the entire codomain by a different weighted category, which could in turn be acted on via an enriched cofunctor to another weighted category, and so forth---suggesting the versatility of comonoids in $\Sig\Pi\cat{C}$ for modeling general interactions.

So far, the examples we have described could have been done in $\Sig\Pi\cat1\approx\poly$ (indeed, the material so far is adapted from \cite{spivak2020poly,polybook}); we have yet to make use of the enriched structure.
Now we will put finite weights in the codomains of our systems to bound their behavior.

\begin{definition}
    A discrete dynamical system $\phi\c\ecat{X}\to\infty$ is \textbf{\textit{bounded (above) by $r\in[0,\infty]$}} if $\phi$ factors through the morphism $r\to\infty$ in $[0,\infty]_\leq\subset\Sig\Pi[0,\infty]_\leq$.
    Equivalently, for each point $x\in\Ob\ecat{X}$, the path $\phi^\sharp_x$ has cost at most $r$. 
\end{definition}

Boundedness is well-behaved under composition: if $\phi\c\ecat{X}\to\infty$ factors through $r$ as $\ol\phi\c\ecat{X}\to r$, then the $n$-fold composition product $\phi\tripow{n}\c\ecat{X}\tripow{n}\to\infty\tripow{n}\iso\infty$ factors through $r\tripow{n}\iso nr$ as $\ol\phi\tripow{n}\c\ecat{X}\tripow{n}\to r\tripow{n}$.
Hence the $n$-fold composite dynamical system $\phi\tripow{n}\circ\delta^{n-1}$ from \eqref{diag.seq} must factor through $nr$ as well, so it is a discrete dynamical system bounded by $nr$.
This coincides with our intuition: if the cost of every time step of a dynamical system is bounded above by $r$, then the cost of $n$ successive time steps must be bounded above by $nr$.
We thus have the following result, generalizing \cref{prop.dyn-sys-cof}.

\begin{proposition} \label{prop.dyn-sys-cof-bound}
    There is a natural correspondence between discrete dynamical systems $\phi\c\ecat{X}\to\infty$ bounded above by $r\in[0,\infty]$ and enriched cofunctors $\Phi\c\ecat{X}\cof\prod_{n\in\nn}nr$, whose codomain is the one-object $(\Sig[0,\infty]_\geq,0,\oplus)$-enriched category with hom-set $\nn$, addition as composition, and weights $|n|\coloneqq nr$.
\end{proposition}
\begin{proof}
    The construction mirrors the one in the proof of \cref{prop.dyn-sys-cof}; we need only verify that the additional restrictions on costs are satisfied.
    Given $\phi$ bounded by $r$, the $n$-fold composite path from \eqref{diag.comp-path} has cost at most $nr$, ensuring $|\Phi^\sharp_x(n)|\leq nr$.
    Conversely, given $\Phi$, we have $|\phi^\sharp_x|=|\Phi^\sharp_x(1)|\leq r$.
\end{proof}

The preceding material is only a sample of how $\Sig\Pi[0,\infty]_\geq$ and, by extension, $\Sig\Pi\cat{C}$ may be used to model compositional behavioral patterns of dynamical systems.
We could generalize the codomain of our discrete dynamical systems beyond one-position, one-direction polynomials in $\cat{C}$; we could generalize $\cat{C}$ beyond mere posets; and so forth.
Indeed, \cref{prop.dyn-sys-cof,prop.dyn-sys-cof-bound} are special cases of a far more general result.

\begin{theorem} \label{thm.cofree}
The forgetful functor from comonoids in $(\Sig\Pi\cat{C},e,\tri)$ to their underlying polynomials has a right adjoint, yielding cofree $(\Sig\cat{C}\op,e,\odot)$-enriched categories on polynomials in $\cat{C}$.
\end{theorem}

\begin{proof}[Sketch of proof]
The construction follows the analogous result for cofree polynomial comonads as detailed in \cite{polybook}.
There the cofree category on a given polynomial has tuples of the polynomial's directions as morphisms; we then assign each tuple a weight in $\cat{C}$ equal to the monoidal product of the predicates of the directions in the tuple.
\end{proof}

\section{Future directions} \label{sec.future}

We close with future directions for research in addition to the potential applications already suggested.

\subsection*{Foundations of polynomial categories}

Spivak surveys categorical properties and structures on $\poly\approx\Sig\Pi\cat1$ in \cite{poly-ref}; in addition to those we have already covered, it would be instructive to examine which of these properties and structures carry over to $\Sig\Pi\cat{C}$, perhaps requiring various conditions on $\cat{C}$.
Similarly, we could investigate how known structures on $\dial(\sets)\approx\hmg(\cat2)$ carry over to $\hmg(\cat{C})$.

\subsection*{Interaction between monoidal structures on polynomials}

Spivak observed that $\tri$ is duoidal over $\otimes$ in the case of $\cat{C}\coloneqq\cat1$, i.e.\ there is a natural transformation $(-\tri-)\otimes(-\tri-)\to(-\otimes-)\tri(-\otimes-)$ satisfying various coherence conditions \cite{spivak2020poly}.
Shapiro and Spivak go on to leverage this duoidality to model compositional dependence \cite{duoidal}.
We hope to generalize their results to the parallel and compositional products on $\Sig\Pi\cat{C}$.

\subsection*{Other monoidal structures on polynomials}

Given a monoidal category $(\cat{C},e,\cdot)$, there are at least two other monoidal structures on $\Sig\Pi\cat{C}$ with unit $e$: one given by
\[
    \Big(\sum_{i\in I}\prod_{a\in A_i}u_{i,a}\Big)\Bowtie\Big(\sum_{j\in J}\prod_{b\in B_j}v_{j,b}\Big)\coloneqq\sum_{i\in I}\sum_{j\in J}\,\prod_{a\c J\to A_i}\,\prod_{b\c I\to B_j}(u_{i,aj}\cdot v_{j,bi})
\]
and another given by
\[
    \Big(\sum_{i\in I}\prod_{a\in A_i}u_{i,a}\Big)\rtimes\Big(\sum_{j\in J}\prod_{b\in B_j}v_{j,b}\Big)\coloneqq\sum_{i\in I}\sum_{j\in J}\,\prod_{a\c J\to A_i}\,\prod_{b\in B_j}(u_{i,aj}\cdot v_{j,b}).
\]
We would like to know if there are interpretations or applications for these monoidal products as there are for the parallel and composition products.

\subsection*{Recovering categories and cofunctors enriched over any distributive category}

\cref{thm} recovers the category of small categories and cofunctors enriched over a free coproduct completion with Day convolution as the category of comonoids of a particular monoidal category.
Yet Clarke and Di Meglio described how cofunctors may be enriched over any distributive monoidal category \cite{enr-cof}.
The free coproduct completion with Day convolution gives us a way to freely construct a distributive monoidal category from any monoidal category, but not every distributive monoidal category arises this way.
We would like to know if our theorem may be generalized to recover categories of small categories and cofunctors enriched over any distributive monoidal category as some category of comonoids.

\newpage

% \appendix

% \section{Additional Proofs}\label{proofs}

% \nocite{*}
\printbibliography
\end{document}